\documentclass[11pt,reqno]{amsart}
\usepackage[T1]{fontenc}
\usepackage[utf8]{inputenc}
\usepackage{amsmath,amssymb,amsthm}
\usepackage[margin=1.15in]{geometry}
\usepackage[protrusion=true,expansion=false]{microtype}
\usepackage[colorlinks=true,linkcolor=blue,citecolor=blue,urlcolor=blue]{hyperref}

\theoremstyle{plain}
\newtheorem{theorem}{Theorem}[section]
\newtheorem{proposition}[theorem]{Proposition}
\newtheorem{lemma}[theorem]{Lemma}
\newtheorem{corollary}[theorem]{Corollary}
\newtheorem{conjecture}[theorem]{Conjecture}
\theoremstyle{definition}
\newtheorem{definition}[theorem]{Definition}
\newtheorem{example}[theorem]{Example}
\newtheorem{remark}[theorem]{Remark}

\DeclareMathOperator{\PF}{PF}
\DeclareMathOperator{\Ap}{Ap}
\DeclareMathOperator{\Min}{Min}
\DeclareMathOperator{\Max}{Max}
\newcommand{\N}{\mathbb{N}}
\newcommand{\Z}{\mathbb{Z}}
\newcommand{\mult}{\mathrm{m}}
\newcommand{\FF}{\mathrm{F}}
\newcommand{\cc}{\mathrm{c}}
\newcommand{\gnus}{\mathrm{g}}
\newcommand{\nn}{\mathrm{n}}
\newcommand{\ee}{\mathrm{e}}
\newcommand{\typ}{\mathrm{t}}
\newcommand{\W}{\mathrm{W}}
\newcommand{\LL}{\mathrm{L}}
\newcommand{\Lam}{\Lambda}
\newcommand{\Xii}{\Xi}
\newcommand{\lam}{\lambda}

\begin{document}

\title[The type of a numerical semigroup and Wilf's conjecture]
{An upper bound for the type of a numerical semigroup,\\
and a reduction of Wilf's conjecture}

\author{Mohammad F. Marashdeh}
\address{Department of Mathematics, Mutah University, Karak, Jordan}
\email{marashdeh@mutah.edu.jo}


\begin{abstract}
Let $S$ be a numerical semigroup with multiplicity $\mult$, conductor $\cc$, embedding
dimension $\ee$, type $\typ$ and genus $\gnus$, and let $\nn=\cc-\gnus$. Wilf's
conjecture asserts that $\ee\,\nn\ge\cc$; the inequality $\gnus\le\typ\,\nn$ of
Fr\"oberg, Gottlieb and H\"aggkvist settles it when $\typ\le\ee-1$. The Ap\'ery set of
$S$ with respect to any $s\in S\setminus\{0\}$ carries a partial order whose maximal
elements are the pseudo-Frobenius numbers translated by $s$; for $s=\mult$ its minimal
elements are the minimal generators other than $\mult$. Comparing the two extremal
statistics bounds the type by $\typ\le\ee-1+\Xii(S)\le\ee-1+\Theta(S)$, where
$\Theta(S)$ measures the redundancy of the covering of the gaps of $S$ by the
pseudo-Frobenius numbers and $\Xii(S)$ refines it. With an exact decomposition of the
Wilf number this yields the genus bound $\gnus\le\ee-1+\typ(\nn-1)$, strictly stronger
than $\gnus\le\typ\,\nn$ precisely when $\typ\ge\ee$, and reduces Wilf's conjecture to
an inequality free of $\cc$ and $\nn$. We determine the equality case of
$\gnus\le\typ\,\nn$, recovering a classification of Singhal; answer a question of
Moscariello and Sammartano whenever $\ee\ge\typ+1$; and correct Kaplan's
classification of the equality case for $\cc\le2\mult$, from which an infinite family
is missing.
\end{abstract}

\maketitle

\section{Introduction}

A \emph{numerical semigroup} is a submonoid $S\subseteq\N=\{0,1,2,\dots\}$ whose
complement $\N\setminus S$ is finite. The elements of $\N\setminus S$ are the
\emph{gaps} of $S$, their number $\gnus=\gnus(S)$ is the \emph{genus}, the largest gap
$\FF=\FF(S)$ is the \emph{Frobenius number}, and $\cc=\cc(S)=\FF+1$ is the
\emph{conductor}. We write
\[
\LL=\LL(S)=S\cap[0,\cc),\qquad \nn=\nn(S)=|\LL(S)|,
\]
so that $\nn+\gnus=\cc$. Every numerical semigroup has a unique minimal system of
generators, whose cardinality $\ee=\ee(S)$ is the \emph{embedding dimension}, and we
denote by $\mult=\mult(S)$ the \emph{multiplicity}, that is, the least positive
element of $S$. Standard references are the monographs of Rosales and
Garc\'ia-S\'anchez \cite{RGS} and of Ram\'irez Alfons\'in \cite{RA}.

In 1978 Wilf \cite{Wilf} posed a question that remains one of the central open
problems in the theory.

\begin{conjecture}[Wilf]\label{conj:wilf}
For every numerical semigroup $S\neq\N$,
\[
\ee(S)\,\nn(S)\;\ge\;\cc(S).
\]
\end{conjecture}

We record the discrepancy as the \emph{Wilf number}
\[
\W(S)\;=\;\ee(S)\,\nn(S)-\cc(S),
\]
so that Conjecture \ref{conj:wilf} reads $\W(S)\ge0$.

The conjecture has been established in the following cases: $\ee\le3$, due to Fr\"oberg,
Gottlieb and H\"aggkvist \cite{FGH}
and to Dobbs and Matthews \cite{DM}; $\cc\le 2\mult$, due to Kaplan \cite{Kaplan};
$2\ee\ge\mult$, due to Sammartano \cite{Sammartano}; $\cc\le3\mult$ and $\nn\le10$,
due to Eliahou \cite{EliahouJEMS}; $3\ee\ge\mult$, due to Eliahou \cite{EliahouEJC};
$\nn\le12$, due to Eliahou and Mar\'in-Arag\'on \cite{EliahouMarin}; semigroups with a
large second generator, due to Spirito \cite{Spirito}; semigroups with $\mult\mid\cc$
and $4\ee\ge\mult$, due to Eliahou \cite{EliahouDivsets}; almost symmetric
semigroups, due to Barucci \cite{Barucci} and, by a different route, to D'Anna and
Moscariello \cite{DM23}; multiplicity $\mult\le19$, due to Bruns,
Garc\'ia-S\'anchez, O'Neill and Wilburne \cite{BGOW} and Kliem and Stump \cite{KS}; and
genus at most $100$, verified by Delgado, Eliahou and Fromentin \cite{DEF}.

Numerical semigroups with small or negative Wilf number, and the near-misses among
them, are studied by Eliahou and Fromentin \cite{EliahouFromentin} and by Delgado
\cite{DelgadoMathZ}. Combined with Zhai's asymptotic count \cite{Zhai}, Eliahou's
theorem on $\cc\le3\mult$ shows that the conjecture holds for asymptotically all
numerical semigroups ordered by genus. We refer to Delgado's survey
\cite{DelgadoSurvey} for a systematic account.

\subsection{The type and the FGH range}

An integer $f$ is a \emph{pseudo-Frobenius number} of $S$ if $f\notin S$ and
$f+s\in S$ for every $s\in S\setminus\{0\}$. The set of these is denoted $\PF(S)$
and its cardinality is the \emph{type} $\typ=\typ(S)$ of $S$. The classical approach to
Conjecture \ref{conj:wilf} rests on the inequality of Fr\"oberg, Gottlieb
and H\"aggkvist \cite[Theorem 20]{FGH},
\begin{equation}\label{eq:FGH}
\cc(S)\;\le\;\bigl(\typ(S)+1\bigr)\,\nn(S),
\end{equation}
which yields Conjecture \ref{conj:wilf} whenever
\begin{equation}\label{eq:FGHrange}
\typ(S)\;\le\;\ee(S)-1 .
\end{equation}
We refer to \eqref{eq:FGHrange}, equivalently to $\ee\ge\typ+1$, as the
\emph{FGH range}. It contains all symmetric semigroups and all semigroups of maximal embedding
dimension; see \cite[Chs.~2--3]{RGS}.

Outside the FGH range, the type has been of limited use. As Delgado observes of
\eqref{eq:FGH}, ``since no upper bound exists for type in higher embedding dimensions,
Proposition 3.1 [that is, \eqref{eq:FGH}] cannot yield further general results''
\cite[\S3.1]{DelgadoSurvey}.
Large classes containing many semigroups with $\typ\ge\ee$ have been settled by other
means, almost symmetric semigroups among them. To the best of our knowledge, however,
no theorem in the literature takes $\typ\ge\ee$ as its hypothesis, and none of the
known cases of Conjecture \ref{conj:wilf} is proved by an argument that runs through
the type in that range. We show that the classical argument behind
\eqref{eq:FGH} yields more than the inequality itself, and that the additional
structure identifies the obstruction.

\subsection{Results}

The proof of \eqref{eq:FGH} discards two nonnegative quantities. For an integer
$x\in[1,\FF]$ let $\nu_S(x)=|\{f\in\PF(S):f-x\in S\}|$ be the number of
pseudo-Frobenius numbers dominating $x$, and set
\[
\lam_S(y)=\bigl|S\cap[0,y]\bigr|,\qquad
\sigma(S)=\sum_{f\in\PF(S)}\bigl|S\cap(f,\FF]\bigr|,\qquad
\Theta(S)=\sum_{x\in\N\setminus S}\bigl(\nu_S(x)-1\bigr).
\]
Here $\sigma(S)$, the dominance defect, measures the failure of the pseudo-Frobenius
numbers to dominate $\LL(S)$, while $\Theta(S)$, the redundancy defect, measures the
redundancy of the covering of the gap set on which the proof of \eqref{eq:FGH} rests.
The resulting decomposition of the Wilf number (Proposition \ref{prop:defect}) is
exact:
\begin{equation}\label{eq:intro-defect}
\W(S)\;=\;\bigl(\ee-\typ-1\bigr)\nn\;+\;\sigma(S)\;+\;\Theta(S),
\qquad \sigma(S),\ \Theta(S)\ \ge\ 0 ,
\end{equation}
and \eqref{eq:FGH} is what remains once both defects are dropped. Dropping $\Theta$
alone sharpens it to $\cc\le(\typ+1)\nn-\sigma(S)$, with equality if and only if
$\Theta(S)=0$ (Corollary \ref{cor:FGH}).

Our principal result is an upper bound for the type. The invariants occurring in
\eqref{eq:FGH} and in Conjecture \ref{conj:wilf} are the two extremal statistics of a
single finite poset, the Ap\'ery poset
$Q(S)=(\Ap(S,\mult)\setminus\{0\},\preceq_S)$: its minimal elements are the primitive
elements (minimal generators) of $S$ other than $\mult$, and its maximal elements are
$\PF(S)+\mult$ (Lemma \ref{lem:poset}). The base point need not be the multiplicity:
Lemma \ref{lem:poset} is proved for $\Ap(S,s)$ with
$s\in S\setminus\{0\}$ arbitrary, giving one inequality for each element of $S$
(Proposition \ref{prop:general}); the multiplicity is the unique base point $s$ for
which every minimal element $p$ satisfies $p>s$, so that $p-s$ is a gap. Bounding the
number of maximal elements by the number of minimal ones then gives the following, in
which
$\Xii(S)=\sum_{w\in\Ap(S,\mult)\setminus\{0\}}\lfloor w/\mult\rfloor(\nu_S(w-\mult)-1)$.

\begin{theorem}[cf. Theorem \ref{thm:type}]\label{thm:intro-type}
For every numerical semigroup $S\ne\N$,
\[
\typ(S)\;\le\;\sum_{p\in P\setminus\{\mult\}}\nu_S(p-\mult)
\;\le\;\ee(S)-1+\Xii(S)\;\le\;\ee(S)-1+\Theta(S),
\]
where $P$ is the set of primitive elements of $S$.
\end{theorem}

The weighted term $\Xii(S)$ rests on one further observation: $\nu_S$ is
nondecreasing
along the descending chain $x,x-\mult,x-2\mult,\dots$ of gaps (Lemma
\ref{lem:numono}), so the redundancy counted by $\Theta(S)$ accumulates at the foot of
each residue class. The refinement is strict for all but $24{,}094$ of the
$23{,}663{,}125$ numerical semigroups of genus at most $31$. Computation supports the
same bound with $\sigma(S)$ in place of $\Xii(S)$ (Conjecture \ref{conj:sigma}).

Thus a numerical semigroup lies outside the FGH range only if
$\Theta(S)\ge\typ+1-\ee$. Rewritten by means of \eqref{eq:intro-defect}, Theorem
\ref{thm:intro-type} becomes an unconditional bound for the genus (Corollary
\ref{cor:master}),
\[
\gnus(S)+\sigma(S)\;\le\;\ee(S)-1+\typ(S)\bigl(\nn(S)-1\bigr),
\]
whose weaker form $\gnus\le\ee-1+\typ(\nn-1)$, obtained by discarding $\sigma$, is
strictly stronger than \eqref{eq:FGH} if and only if $\typ\ge\ee$, that is, exactly
outside the range in which \eqref{eq:FGH} settles Conjecture \ref{conj:wilf}.

We do not prove Wilf's conjecture, and neither bound above settles a case of it that
\eqref{eq:FGH} does not: Theorem \ref{thm:intro-type} bounds $\Theta(S)$ from below by
$\typ+1-\ee$, whereas by \eqref{eq:intro-defect} the conjecture requires a bound
larger by a factor of $\nn$.

The same identity does, however, reduce Conjecture \ref{conj:wilf} to a statement in
which neither $\cc$ nor $\nn$ appears. Write $\PF(S)=\{f_1<\dots<f_{\typ}\}$ and
$d=\typ+1-\ee$. Only $d$ of the $\typ$ nonnegative summands of $\sigma(S)$ are needed
to absorb the factor $\nn$, and one obtains unconditionally (Proposition
\ref{prop:reduction}) that $\W(S)\ge\Theta(S)-\sum_{i=1}^{d}\lam_S(f_i)$ whenever
$\typ\ge\ee$. Wilf's conjecture therefore follows from the inequality
$\Theta(S)\ge\sum_{i=1}^{d}\lam_S(f_i)$ on the range $\typ\ge\ee$ (Conjecture
\ref{conj:R} and Proposition \ref{prop:Rwilf}); we verify that inequality over all
$171{,}202{,}689$ numerical semigroups of genus at most $35$.

The equality case $\W(S)=0$ was posed by Wilf \cite{Wilf}.
Two families are classical: every $2$-generated semigroup is symmetric and therefore
satisfies $\W(S)=0$; and for $\mult\ge2$, $k\ge1$ the semigroup
\begin{equation}\label{eq:intro-family}
T_{\mult,k}\;=\;\{0,\mult,2\mult,\dots,(k-1)\mult\}\cup[k\mult,\infty)
\;=\;\langle \mult,\ k\mult+1,\ k\mult+2,\ \dots,\ k\mult+\mult-1\rangle
\end{equation}
has $\cc=k\mult$, $\nn=k$, $\ee=\mult$ and $\W(T_{\mult,k})=0$. Moscariello and
Sammartano \cite[Question 8]{MS} asked whether these are the only ones; the question
is reproduced as \cite[Problem 2.5]{DelgadoSurvey}. It has been answered
affirmatively for $\cc\le3\mult$
\cite[Remark 6.6]{EliahouJEMS} and, by computation, for $\mult\le15$
\cite[Remark 4.7]{BGOW}. We settle it on the whole FGH range.

\begin{theorem}[cf. Theorem \ref{thm:equality}]\label{thm:intro-equality}
Let $S\neq\N$ be a numerical semigroup with $\ee(S)\ge\typ(S)+1$. Then $\W(S)=0$ if
and only if $\ee(S)=2$ or $S=T_{\mult,k}$ for some $\mult\ge3$, $k\ge1$.
\end{theorem}

The proof shows, via \eqref{eq:intro-defect}, that $\W(S)=0$ and $\ee\ge\typ+1$ force
$\ee=\typ+1$ and $\sigma(S)=\Theta(S)=0$, and then classifies the semigroups with
$\sigma=\Theta=0$, that is, the equality case of \eqref{eq:FGH}. That classification
is due to Singhal \cite[Theorem 1.6]{Singhal}. The FGH range is a condition on $\ee$
and $\typ$ alone, so Theorem \ref{thm:intro-equality} imposes no bound on the
multiplicity; the theorem covers, for instance, every semigroup of maximal embedding
dimension. We conjecture that the hypothesis $\ee\ge\typ+1$ can be removed altogether
(Conjecture \ref{conj:final}).

Finally, we correct \cite[Proposition 26]{Kaplan}. Kaplan proved Wilf's conjecture for
$\FF<2\mult$, that is, for $\cc\le2\mult$ \cite[Theorem 24]{Kaplan}, and asserted that
in this range equality holds only for
$\langle \mult,\mult+1,\dots,2\mult-1\rangle$ and $\langle3,4\rangle$. An infinite
family is missing.

\begin{theorem}[= Theorem \ref{thm:kaplan}]\label{thm:intro-kaplan}
Let $S\neq\N$ be a numerical semigroup with $\cc(S)\le2\mult(S)$. Then $\W(S)=0$ if
and only if $S=T_{\mult,1}$ or $S=T_{\mult,2}$ for some $\mult\ge2$, or
$S=\langle3,4\rangle$.
\end{theorem}

The smallest members of the missing family $T_{\mult,2}=\langle
\mult,2\mult+1,\dots,3\mult-1\rangle$ are $\langle2,5\rangle$ and
$\langle3,7,8\rangle=\{0,3\}\cup[6,\infty)$; the latter has $\mult=3$,
$\FF=5<6=2\mult$, $\nn=2$, $\ee=3$ and $\W=3\cdot2-6=0$. Remark \ref{rmk:kaplan}
identifies the step at which \cite[Proposition 26]{Kaplan} fails. Our proof of Theorem
\ref{thm:intro-kaplan} reduces to a statement about perfect additive bases of order
$2$ (Lemma \ref{lem:perfect}).

Sections \ref{sec:poset}--\ref{sec:equality} branch from Section \ref{sec:defect}:
Section \ref{sec:poset} exploits the redundancy defect, Section \ref{sec:reduction}
the dominance defect, and Section \ref{sec:equality}, through the identity
\eqref{eq:intro-defect}, the vanishing of both. Section \ref{sec:kaplan} is
independent of all three, resting only on the elementary Lemma \ref{lem:family}.
Section \ref{sec:computation} reports the verification and Section
\ref{sec:conclusion} states what remains.

\section{Notation and preliminaries}\label{sec:prelim}

This section fixes notation, records the standard facts used throughout, and computes
the invariants of a family of numerical semigroups that recurs below.

Throughout this paper, $S$ denotes a numerical semigroup and $S^{*}=S\setminus\{0\}$.
We abbreviate $\mult=\mult(S)$, $\FF=\FF(S)$, $\cc=\cc(S)$, $\gnus=\gnus(S)$,
$\LL=\LL(S)$, $\nn=\nn(S)$, $\ee=\ee(S)$, $\typ=\typ(S)$ whenever $S$ is clear from
the context, and we write $G(S)=\N\setminus S$ for the gap set. We exclude
$S=\N$ throughout, so that $\mult\ge2$ and $\FF\ge1$. By definition of the conductor,
$[\cc,\infty)\subseteq S$.

An element $x\in S^{*}$ is \emph{decomposable} if $x=y+z$ with $y,z\in S^{*}$, and
\emph{primitive} otherwise; the primitive elements form the unique minimal system of
generators of $S$, so $\ee$ is their number. If $x\ge\cc+\mult$ then $x-\mult\ge\cc$, so
$x=\mult+(x-\mult)$ is decomposable; all primitives therefore lie in
$[\mult,\cc+\mult)$.

On $\Z$ we use the partial order $a\preceq_S b\iff b-a\in S$, and we write
$a\prec_S b$ when $a\preceq_S b$ and $a\ne b$. The following facts are standard; see
\cite[Ch.~2]{RGS}. We include proofs of (vii) and of (ii), the
latter being used repeatedly below.

\begin{lemma}\label{lem:standard}
Let $S\neq\N$ be a numerical semigroup.
\begin{enumerate}
\item[(i)] $\PF(S)$ is precisely the set of maximal elements of $\bigl(G(S),\preceq_S\bigr)$;
in particular $\FF\in\PF(S)$ and $\typ\ge1$.
\item[(ii)] For every $x\in G(S)$ there exists $f\in\PF(S)$ with $f-x\in S$.
\item[(iii)] $\typ(S)=1$ if and only if $S$ is symmetric, that is, $\gnus=\cc/2$.
\item[(iv)] For every $s\in S^{*}$, writing
$\Ap(S,s)=\{u\in S:\ u-s\notin S\}$ for the \emph{Ap\'ery set} of $S$ with respect to
$s$, one has $\PF(S)=\{w-s:\ w\ \text{maximal in }(\Ap(S,s),\preceq_S)\}$.
\item[(v)] For every $s\in S^{*}$, $|\Ap(S,s)|=s$, and $\Ap(S,s)$ contains exactly one
element of each residue class modulo $s$, namely the least element of $S$ in that
class; in particular $0$ is its only element divisible by $s$, and
$s\notin\Ap(S,s)$.
\item[(vi)] $\max\Ap(S,s)=\FF+s$ for every $s\in S^{*}$.
\item[(vii)] $\ee(S)\ge2$; and if $\ee(S)=2$ then $S$ is symmetric.
\end{enumerate}
\end{lemma}

\begin{proof}[Proof of \textup{(vii)} and \textup{(ii)}]
For (vii): if $\ee(S)=1$ then $S=\langle\mult\rangle$, whose complement in $\N$ is
infinite unless $\mult=1$, that is, unless $S=\N$. The symmetry of
$S=\langle a,b\rangle$ with $\gcd(a,b)=1$ is Sylvester's classical computation of
$\FF=ab-a-b$ and $\gnus=(a-1)(b-1)/2$; see \cite[Ch.~3]{RA}.

For (ii), let $x\in G(S)$ and let $Y=\{y\in G(S):\ y-x\in S\}$. Then $x\in Y$, so
$Y\ne\emptyset$, and $Y$ is finite because $Y\subseteq G(S)$. Let $f=\max Y$. If
$s\in S^{*}$ and $f+s\notin S$,
then $f+s\in G(S)$ and $(f+s)-x=(f-x)+s\in S$, so $f+s\in Y$, contradicting
maximality. Thus $f$ is a maximal element of $(G(S),\preceq_S)$; by (i), $f\in\PF(S)$,
and $f-x\in S$.
\end{proof}

The decomposition of the next section counts the incidences between pseudo-Frobenius
numbers and elements of $S$. We record the two properties of that incidence set which
the count requires.

\begin{lemma}\label{lem:pairs}
Let $\Pi=\{(f,s)\in\PF(S)\times S:\ s\le f\}$.
\begin{enumerate}
\item[(i)] For every $(f,s)\in\Pi$ one has $f-s\in G(S)$.
\item[(ii)] The map $\pi:\Pi\to G(S)$, $\pi(f,s)=f-s$, is surjective.
\end{enumerate}
\end{lemma}

\begin{proof}
(i) If $s=0$ then $f-s=f\in G(S)$. If $s>0$ and $f-s\in S$, then $f=(f-s)+s\in S$,
contradicting $f\in G(S)$. (ii) is Lemma \ref{lem:standard}(ii).
\end{proof}

Finally we record the invariants of the family \eqref{eq:intro-family}.

\begin{lemma}\label{lem:family}
Let $\mult\ge2$ and $k\ge1$, and let
$T_{\mult,k}=\{0,\mult,\dots,(k-1)\mult\}\cup[k\mult,\infty)$. Then $T_{\mult,k}$ is
a numerical semigroup with
\[
\mult(T_{\mult,k})=\mult,\quad \cc=k\mult,\quad \nn=k,\quad
\ee=\mult,\quad \typ=\mult-1,\quad \W(T_{\mult,k})=0,
\]
and $T_{\mult,k}=\langle \mult,\,k\mult+1,\,k\mult+2,\dots,k\mult+\mult-1\rangle$.
\end{lemma}

\begin{proof}
Closure under addition is clear: a sum of two multiples of $\mult$ that is less than
$k\mult$ is again such a multiple, and any sum involving an element $\ge k\mult$ is
$\ge k\mult$. As $k\mult$ and $k\mult+1$ both belong to $T_{\mult,k}$, the set
generates $\Z$, so $T_{\mult,k}$ is a numerical semigroup; its least positive element
is $\mult$, and $k\mult-1\notin T_{\mult,k}$ because $\mult\ge2$, whence $\cc=k\mult$
and $\LL=\{0,\mult,\dots,(k-1)\mult\}$ has $k$ elements.

We determine the primitives. Every $x\ge k\mult+\mult$ is decomposable, since
$x=\mult+(x-\mult)$ with $x-\mult\ge k\mult$. Let $x=k\mult+r$ with $0\le r<\mult$ and
suppose $x=a+b$ with $a,b\in T_{\mult,k}^{*}$. If both $a,b<k\mult$ then both are
multiples of $\mult$, so $\mult\mid x$ and hence $r=0$; if instead $a\ge k\mult$, then
$x=a+b\ge k\mult+\mult$, a contradiction. Thus the only element of
$[k\mult,k\mult+\mult)$ that can be
decomposable is $k\mult$ itself, which is decomposable exactly when
$k\ge2$ (as $k\mult=\mult+(k-1)\mult$). Among the elements $i\mult$ with
$1\le i\le k-1$, those with $i\ge2$ are decomposable and $\mult$ is not. Hence the
primitives are $\mult$ together with $k\mult+1,\dots,k\mult+\mult-1$, so $\ee=\mult$.

For the type, the gaps of $T_{\mult,k}$ are exactly the integers
$i\mult+j$ with $0\le i\le k-1$ and $1\le j\le\mult-1$. If $i=k-1$, then for every
$s\in T_{\mult,k}^{*}$ we have $s\ge\mult$ and therefore
$(k-1)\mult+j+s\ge k\mult+j>k\mult$, so $(k-1)\mult+j\in\PF(T_{\mult,k})$. If
$i\le k-2$, then $(i\mult+j)+\mult=(i+1)\mult+j$ with $i+1\le k-1$ and
$1\le j\le\mult-1$, which is again a gap, so $i\mult+j\notin\PF(T_{\mult,k})$.
Hence $\PF(T_{\mult,k})=\{(k-1)\mult+j:1\le j\le\mult-1\}$, so $\typ=\mult-1$ and
$\W=\mult\cdot k-k\mult=0$.
\end{proof}

For $\mult=2$ the family degenerates to $T_{2,k}=\langle2,2k+1\rangle$, which is
$2$-generated and hence symmetric; this is why the case $\mult=2$ is absorbed into the
symmetric case in the classifications proved below.

\section{A defect decomposition of the Wilf number}\label{sec:defect}

The Fr\"oberg--Gottlieb--H\"aggkvist argument bounds the genus by covering the gap set
with the reflected sets $f-\LL$, $f\in\PF(S)$, and discarding two sources of loss: the
elements of $\LL$ that exceed a given pseudo-Frobenius number, and the gaps that the
covering reaches more than once. Restoring both turns \eqref{eq:FGH} into an identity
for the Wilf number.

\begin{definition}\label{def:defects}
Let $S\neq\N$ be a numerical semigroup. For $y\in\Z$ set
$\lam_S(y)=\bigl|S\cap[0,y]\bigr|$, and for every integer $x\in[1,\FF]$ set
\[
\nu_S(x)=\bigl|\{f\in\PF(S):\ f-x\in S\}\bigr| .
\]
Define
\[
\Lam(S)=\sum_{f\in\PF(S)}\lam_S(f),\qquad
\sigma(S)=\sum_{f\in\PF(S)}\bigl|S\cap(f,\FF]\bigr|,\qquad
\Theta(S)=\sum_{x\in G(S)}\bigl(\nu_S(x)-1\bigr),
\]
and call $\sigma(S)$ the \emph{dominance defect} and $\Theta(S)$ the
\emph{redundancy defect} of $S$.
\end{definition}

By Lemma \ref{lem:pairs}(i) one has $\nu_S(x)=0$ for $x\in S\cap[1,\FF]$, and by
Lemma \ref{lem:standard}(ii) one has $\nu_S(x)\ge1$ for $x\in G(S)$; hence
$\Theta(S)\ge0$.

Since $\PF(S)\subseteq[1,\FF]$ and $S\cap[0,f]\subseteq\LL$ for $f\le\FF$, we have
$\lam_S(f)=\bigl|\LL\cap[0,f]\bigr|\le \nn$ for every $f\in\PF(S)$, and therefore
\begin{equation}\label{eq:sigmaLambda}
\sigma(S)=\sum_{f\in\PF(S)}\bigl(\nn-\lam_S(f)\bigr)=\typ\,\nn-\Lam(S)\ \ge\ 0 ,
\end{equation}
the middle equality because $S\cap[0,f]$ and $S\cap(f,\FF]$ partition
$S\cap[0,\FF]=\LL$ for $f\in\PF(S)$.

\begin{proposition}[Defect decomposition]\label{prop:defect}
For every numerical semigroup $S\neq\N$,
\begin{equation}\label{eq:gcount}
\gnus(S)\;=\;\Lam(S)-\Theta(S),
\end{equation}
and consequently
\begin{equation}\label{eq:defect}
\W(S)\;=\;\bigl(\ee(S)-\typ(S)-1\bigr)\,\nn(S)\;+\;\sigma(S)\;+\;\Theta(S).
\end{equation}
\end{proposition}

\begin{proof}
Let $\Pi$ be as in Lemma \ref{lem:pairs}; it is finite. Counting $\Pi$ by its first
coordinate gives $|\Pi|=\sum_{f\in\PF(S)}\bigl|S\cap[0,f]\bigr|=\Lam(S)$. By Lemma
\ref{lem:pairs} the map $\pi:\Pi\to G(S)$, $\pi(f,s)=f-s$, is well defined and
surjective. Its fibre over a gap $x$ is $\{(f,f-x):\ f\in\PF(S),\ f-x\in S\}$: a pair
$(f,s)\in\Pi$ satisfies $f-s=x$ if and only if $s=f-x$, and then $0\le s\le f$
automatically since $x\ge1$. As $1\le x\le\FF$, that fibre has cardinality $\nu_S(x)$.
Hence
\[
\Lam(S)=|\Pi|=\sum_{x\in G(S)}\nu_S(x)=\gnus(S)+\sum_{x\in G(S)}\bigl(\nu_S(x)-1\bigr)
=\gnus(S)+\Theta(S),
\]
which is \eqref{eq:gcount}. For \eqref{eq:defect}, use $\cc=\nn+\gnus$ and
\eqref{eq:sigmaLambda}:
\[
\W(S)=\ee\,\nn-\cc=\ee\,\nn-\nn-\gnus=(\ee-1)\nn-\Lam(S)+\Theta(S)
=(\ee-1)\nn-\bigl(\typ\,\nn-\sigma(S)\bigr)+\Theta(S),
\]
which is $(\ee-\typ-1)\nn+\sigma(S)+\Theta(S)$.
\end{proof}

\begin{corollary}\label{cor:FGH}
For every numerical semigroup $S\ne\N$,
\begin{equation}\label{eq:sharp}
\cc(S)\;\le\;\bigl(\typ(S)+1\bigr)\nn(S)-\sigma(S)\;\le\;\bigl(\typ(S)+1\bigr)\nn(S).
\end{equation}
Equality holds in the first inequality if and only if $\Theta(S)=0$, and in the second
if and only if $\sigma(S)=0$. In particular $\typ\le\ee-1$ implies $\cc\le\ee\,\nn$.
\end{corollary}

\begin{proof}
By \eqref{eq:gcount} and \eqref{eq:sigmaLambda},
$\cc=\nn+\gnus=\nn+\Lam-\Theta=(\typ+1)\nn-\sigma-\Theta$, and both $\sigma$ and
$\Theta$ are nonnegative.
\end{proof}

The outer inequality in \eqref{eq:sharp} is \eqref{eq:FGH}, equivalently
$\gnus\le\typ\,\nn$; the first inequality strictly refines it whenever $\sigma(S)>0$.
Explicitly, \eqref{eq:gcount} and \eqref{eq:sigmaLambda} say that
\begin{equation}\label{eq:sandwich}
\gnus(S)\;\le\;\Lam(S)\;\le\;\typ(S)\,\nn(S),\qquad
\Theta(S)=\Lam(S)-\gnus(S),\quad \sigma(S)=\typ(S)\,\nn(S)-\Lam(S).
\end{equation}
The lower bound is the surjectivity of $\pi$, and the upper bound is
$\lam_S(f)\le\nn$; \eqref{eq:FGH} is their combination. The defects
$\sigma$ and $\Theta$ measure distinct phenomena: $\sigma(S)=0$ says that
every pseudo-Frobenius number exceeds all elements of $\LL(S)$, whereas
$\Theta(S)=0$ says that the covering of $G(S)$ furnished by Lemma
\ref{lem:standard}(ii) is a partition.

\begin{example}\label{ex:small}
Let $S=\langle3,5,7\rangle$. Since $5,6=3+3$ and $7$ lie in $S$ and $s\in S$ implies
$s+3\in S$, every integer $\ge5$ lies in $S$; as $1,2,4\notin S$, this gives
$S=\{0,3\}\cup[5,\infty)$, so $\gnus=3$, $\FF=4$, $\cc=5$, $\LL=\{0,3\}$, $\nn=2$,
$\ee=3$ and $\PF(S)=\{2,4\}$, whence $\typ=2$. Now $\lam_S(2)=1$ and $\lam_S(4)=2$,
so $\Lam(S)=3$, $\sigma(S)=\typ\,\nn-\Lam(S)=1$ and $\Theta(S)=\Lam(S)-\gnus=0$ by
\eqref{eq:gcount}. Formula \eqref{eq:defect} gives $\W(S)=(3-2-1)\cdot2+1+0=1$, in
agreement with $\ee\,\nn-\cc=3\cdot2-5=1$.
\end{example}

\section{The Ap\'ery poset and an upper bound for the type}\label{sec:poset}

Throughout this section we write
\[
Q(S)\;=\;\bigl(\Ap(S,\mult)\setminus\{0\},\ \preceq_S\bigr),
\]
the \emph{Ap\'ery poset} of $S$, which has $\mult-1$ elements. Its maximal elements
number $\typ$ and its minimal elements $\ee-1$, so the passage from $\typ$ to $\ee-1$
is a poset-theoretic estimate, whose loss the section bounds.

The construction does not single out the multiplicity. We therefore define the poset
for an arbitrary base point: for $s\in S^{*}$ write
$Q_s(S)=\bigl(\Ap(S,s)\setminus\{0\},\preceq_S\bigr)$, so that $Q(S)=Q_\mult(S)$. For
$s=\mult$ both descriptions below are standard, that of $\Max Q(S)$ being Lemma
\ref{lem:standard}(iv) and that of $\Min Q(S)$ the usual passage from the Ap\'ery set
to the minimal system of generators, see \cite[Ch.~2]{RGS}; since the general base point
costs nothing and the two are used jointly, we prove them together.

\begin{lemma}\label{lem:poset}
Let $S\ne\N$ be a numerical semigroup, let $P$ be its set of primitive elements, and
let $s\in S^{*}$. Then
\[
\Min Q_s(S)=\{p\in P:\ p\ne s,\ p-s\notin S\},\qquad \Max Q_s(S)=\PF(S)+s,
\]
where $\Min$ and $\Max$ denote the sets of minimal and of maximal elements. For
$s=\mult$ this reads
\[
\Min Q(S)=P\setminus\{\mult\},\qquad \Max Q(S)=\PF(S)+\mult,
\]
so that $|\Min Q(S)|=\ee(S)-1$ and $|\Max Q(S)|=\typ(S)$.
\end{lemma}

\begin{proof}
A primitive $p$ with $p\ne s$ and $p-s\notin S$ lies in $\Ap(S,s)$ by definition, and
$p\ne0$; and $s\notin\Ap(S,s)$ by Lemma \ref{lem:standard}(v). So the displayed
set is contained in $\Ap(S,s)\setminus\{0\}$.

Let $w\in\Ap(S,s)\setminus\{0\}$. If $w$ is decomposable, write $w=x+y$ with
$x,y\in S^{*}$. Not both $x$ and $y$ are divisible by $s$: otherwise $s\mid w$,
whereas by Lemma \ref{lem:standard}(v) the only element of $\Ap(S,s)$ divisible by $s$
is $0$. Say $s\nmid x$, and let $x'$ be the element of $\Ap(S,s)$ congruent to $x$
modulo $s$; by Lemma \ref{lem:standard}(v) it is the least element of $S$ in that
class, so $x'\le x$, and $x'\ne0$ since $s\nmid x$. Then $w-x'=(x-x')+y\in S$, because
$x-x'$ is a nonnegative multiple of $s$ and $y\in S$; moreover $w-x'\ge y>0$, so
$x'\prec_S w$ and $w$ is not minimal in $Q_s(S)$. Conversely, if $w$ is not minimal,
then $w=z+(w-z)$ for some $z\in\Ap(S,s)\setminus\{0\}$ with $w-z\in S^{*}$, so $w$ is
decomposable. Hence the minimal elements of $Q_s(S)$ are exactly the primitive
elements lying in $\Ap(S,s)\setminus\{0\}$, which is the asserted set.

The description of $\Max Q_s(S)$ is Lemma \ref{lem:standard}(iv); and $0$ is not
maximal, since $s\ge\mult\ge2$ forces $\Ap(S,s)\ne\{0\}$ and $0\preceq_S w$ for every
$w$. Finally, for $s=\mult$ every primitive $p\ne\mult$ satisfies $p-\mult\notin S$,
since otherwise $p=\mult+(p-\mult)$ would be decomposable; so
$\Min Q(S)=P\setminus\{\mult\}$, a set of $\ee-1$ elements.
\end{proof}

Bounding the number of maximal elements of $Q_s(S)$ by the number of its minimal ones
gives one inequality for every base point.

\begin{proposition}\label{prop:general}
Let $S\ne\N$ be a numerical semigroup. For every $s\in S^{*}$,
\begin{equation}\label{eq:general}
\typ(S)\;\le\;\sum_{p\in\Min Q_s(S)}\bigl|\{f\in\PF(S):\ f+s-p\in S\}\bigr| .
\end{equation}
\end{proposition}

\begin{proof}
The poset $Q_s(S)$ is finite and nonempty, since $|\Ap(S,s)|=s\ge\mult\ge2$ by Lemma
\ref{lem:standard}(v); hence $\Min Q_s(S)\ne\emptyset$. Every element of a finite
poset lies above at least one
minimal element, so we may choose, for each $w\in\Max Q_s(S)$, an element
$\mu(w)\in\Min Q_s(S)$ with $\mu(w)\preceq_S w$. Counting the fibres of $\mu$,
\[
\typ(S)=\bigl|\Max Q_s(S)\bigr|
=\sum_{q\in\Min Q_s(S)}\bigl|\mu^{-1}(q)\bigr|
\;\le\;\sum_{q\in\Min Q_s(S)}\bigl|\{w\in\Max Q_s(S):\ q\preceq_S w\}\bigr|,
\]
using $|\Max Q_s(S)|=\typ(S)$ from Lemma \ref{lem:poset}. For $p\in\Min Q_s(S)$ the
same lemma gives $\{w\in\Max Q_s(S):p\preceq_S w\}=\{f+s:\ f\in\PF(S),\ f+s-p\in S\}$,
whose cardinality is the $p$-th summand of \eqref{eq:general}.
\end{proof}

The base point $s=\mult$ is distinguished: it is the only one for which every minimal
element $p$ of $Q_s(S)$ satisfies $p>s$, since for $s\ne\mult$ the multiplicity itself
lies in $\Min Q_s(S)$ and is smaller than $s$. For $s=\mult$, therefore, $p-s$ is a
gap, the summands of \eqref{eq:general} become values of $\nu_S$, and
\eqref{eq:general} becomes a bound by the redundancy defect. Less than $\Theta(S)$ is
needed, by the following monotonicity.

\begin{lemma}\label{lem:numono}
Let $x\in G(S)$ and $s\in S$ with $x-s\ge1$. Then $x-s\in G(S)$ and
$\nu_S(x-s)\ge\nu_S(x)$.
\end{lemma}

\begin{proof}
If $x-s$ lay in $S$ then $x=(x-s)+s$ would lie in $S$, so $x-s$ is a gap. If
$f\in\PF(S)$ satisfies $f-x\in S$, then $f-(x-s)=(f-x)+s\in S$; thus every
pseudo-Frobenius number counted by $\nu_S(x)$ is counted by $\nu_S(x-s)$.
\end{proof}

The monotonicity of Lemma \ref{lem:numono} lets one weight each element $w$ of the
Ap\'ery set by the number $\lfloor w/\mult\rfloor$ of gaps below it in its residue
class.

\begin{theorem}\label{thm:type}
Let $S\ne\N$ be a numerical semigroup, let $P$ be its set of primitive elements, and
put
\begin{equation}\label{eq:xi}
\Xii(S)\;=\;\sum_{w\in\Ap(S,\mult)\setminus\{0\}}
\Bigl\lfloor \tfrac{w}{\mult}\Bigr\rfloor\bigl(\nu_S(w-\mult)-1\bigr).
\end{equation}
Then
\begin{equation}\label{eq:typebound}
\typ(S)\;\le\;\sum_{p\in P\setminus\{\mult\}}\nu_S(p-\mult)
\;\le\;\ee(S)-1+\Xii(S)\;\le\;\ee(S)-1+\Theta(S).
\end{equation}
\end{theorem}

\begin{proof}
Take $s=\mult$ in Proposition \ref{prop:general}. By Lemma \ref{lem:poset},
$\Min Q(S)=P\setminus\{\mult\}$, and for such a $p$ one has $p-\mult\notin S$ and
$p>\mult$, since $p\in S^{*}$ forces $p\ge\mult$ while $p\ne\mult$. As
$\max\Ap(S,\mult)=\FF+\mult$ by Lemma \ref{lem:standard}(vi), we get
$1\le p-\mult\le\FF$, so $p-\mult$ is a gap of $S$, and
\[
p\preceq_S f+\mult\iff f+\mult-p\in S\iff f-(p-\mult)\in S .
\]
Hence the $p$-th summand of \eqref{eq:general} equals $\nu_S(p-\mult)$, which is the
first inequality in \eqref{eq:typebound}.

For the third, fix $w\in\Ap(S,\mult)\setminus\{0\}$ and put $k=\lfloor
w/\mult\rfloor$. By Lemma \ref{lem:standard}(v), $\mult\nmid w$ and $w$ is the least
element of $S$ in its class modulo $\mult$; hence $w-\mult,w-2\mult,\dots,w-k\mult$
are $k$ distinct gaps of $S$, and by Lemma \ref{lem:numono},
$\nu_S(w-j\mult)\ge\nu_S(w-\mult)$ for $1\le j\le k$. Distinct $w$ give gaps in
distinct residue classes, and every gap arises this way, so
\[
\Theta(S)=\sum_{x\in G(S)}\bigl(\nu_S(x)-1\bigr)
\;\ge\;\sum_{w\in\Ap(S,\mult)\setminus\{0\}}
k_w\bigl(\nu_S(w-\mult)-1\bigr)=\Xii(S),
\]
where $k_w=\lfloor w/\mult\rfloor$; this is the third inequality. For the second,
$P\setminus\{\mult\}\subseteq\Ap(S,\mult)\setminus\{0\}$ and $k_p\ge1$ for such $p$,
while all summands in \eqref{eq:xi} are nonnegative, so
\begin{equation*}
\sum_{p\in P\setminus\{\mult\}}\nu_S(p-\mult)
=(\ee-1)+\sum_{p\in P\setminus\{\mult\}}\bigl(\nu_S(p-\mult)-1\bigr)
\le(\ee-1)+\Xii(S).
\end{equation*}
\end{proof}

Theorem \ref{thm:type} makes $\Xii(S)$, and a fortiori $\Theta(S)$, the necessary cost
of leaving the FGH range: $\typ>\ee-1$ forces $\Xii(S)\ge\typ+1-\ee$. The converse
fails: both quantities may be large while $S$ lies well inside the FGH range.

\begin{example}\label{ex:theta}
Let $S=\{0\}\cup[15,23]\cup[30,\infty)$. This is a numerical semigroup: a sum of two
elements of $[15,23]$ lies in $[30,46]$, and any sum involving an element of
$[30,\infty)$ is again at least $30$. Its gaps are $[1,14]\cup[24,29]$, so
$\gnus=20$, $\FF=29$, $\cc=30$ and $\nn=10$. Every element of $[30,44]$ is a sum of
two elements of $[15,23]$, so the primitives are exactly $15,\dots,23$ and $\ee=9$.
A gap $f$ lies in $\PF(S)$ if and only if $f\ge24$: for $f\in[24,29]$ every
$s\in S^{*}$ satisfies $s\ge15$, so $f+s\ge39$ and $f+s\in S$; whereas for
$f\in[1,14]$ the interval $[f+15,f+23]$ meets $[24,29]$, because $f+15\le29$ and
$f+23\ge24$, so some $s\in[15,23]\subseteq S^{*}$ makes $f+s$ a gap. Hence $\typ=6$.
Finally $\lam_S(f)=\nn=10$ for each of the six $f\in\PF(S)$, so $\Lam(S)=60$ and
$\Theta(S)=\Lam(S)-\gnus=40$ by \eqref{eq:gcount}, while $\typ=6\le8=\ee-1$.
\end{example}

\begin{corollary}\label{cor:theta0}
Let $S\ne\N$ be a numerical semigroup with $\Theta(S)=0$. Then $\typ(S)\le\ee(S)-1$,
and consequently $S$ satisfies Wilf's conjecture.
\end{corollary}

\begin{proof}
The first assertion is \eqref{eq:typebound}; the second then follows from
\eqref{eq:sharp}, whose last clause gives $\cc\le\ee\,\nn$ when $\typ\le\ee-1$.
\end{proof}

By Corollary \ref{cor:FGH} the hypothesis $\Theta(S)=0$ is equality in the first
inequality of \eqref{eq:sharp}, so the content of Corollary \ref{cor:theta0} is that
this locus lies inside the FGH range, where Conjecture \ref{conj:wilf} was already
known.

A semigroup outside the FGH range illustrates the whole of \eqref{eq:typebound}.

\begin{example}\label{ex:poset}
Let $S=\langle9,10,12,13\rangle$, so that $\mult=9$, $\LL=\{0,9,10,12,13\}$,
$\nn=5$, $\cc=18$ and $\gnus=13$. Here
\[
\Ap(S,9)=\{0,10,12,13,20,23,24,25,26\},
\]
and $Q(S)$ is this set with $0$ removed. Its minimal elements are $10,12,13$, which
are precisely the primitive elements of $S$ other than $\mult$, so $|\Min Q(S)|=3=\ee-1$;
its maximal elements are $20,23,24,25,26$, and subtracting $\mult$ gives
$\PF(S)=\{11,14,15,16,17\}$, so $|\Max Q(S)|=5=\typ$. Thus $\typ>\ee-1$: this
semigroup lies outside the FGH range, and accordingly both defects are positive,
$\sigma(S)=2$ and $\Theta(S)=10$. The three gaps $p-\mult$ attached to the minimal
elements are $1,3,4$, with $\nu_S(1)=2$, $\nu_S(3)=2$ and $\nu_S(4)=3$. The remaining
elements $20,23,24,25,26$ of $\Ap(S,9)$ have $\lfloor w/9\rfloor=2$ and
$\nu_S(w-9)=1$, so they contribute nothing to $\Xii(S)$, whereas $10,12,13$
contribute $1,1,2$; thus $\Xii(S)=4$ and \eqref{eq:typebound} reads
\[
5\;\le\;2+2+3=7\;\le\;3+4=7\;\le\;3+10=13 .
\]
Here the second inequality in \eqref{eq:typebound} is an equality, while
$\ee-1+\Theta$ exceeds the middle term by $6$.
\end{example}

Substituting \eqref{eq:sandwich} into Theorem \ref{thm:type} turns it into a bound on
the genus.

\begin{corollary}\label{cor:master}
For every numerical semigroup $S\ne\N$,
\begin{equation}\label{eq:master}
\gnus(S)+\sigma(S)\;\le\;\ee(S)-1+\typ(S)\bigl(\nn(S)-1\bigr).
\end{equation}
In particular $\gnus(S)\le\ee(S)-1+\typ(S)\bigl(\nn(S)-1\bigr)$, and this bound is
strictly stronger than the Fr\"oberg--Gottlieb--H\"aggkvist bound
$\gnus(S)\le\typ(S)\,\nn(S)$ if and only if $\typ(S)\ge\ee(S)$.
\end{corollary}

\begin{proof}
By Theorem \ref{thm:type}, $\typ\le\ee-1+\Theta(S)$, and $\Theta(S)=\Lam(S)-\gnus$ by
\eqref{eq:sandwich}; hence $\gnus\le\Lam(S)+\ee-1-\typ$. Substituting
$\Lam(S)=\typ\,\nn-\sigma(S)$, again from \eqref{eq:sandwich}, gives
\eqref{eq:master}, and
the second inequality follows from $\sigma\ge0$. Finally
$\ee-1+\typ(\nn-1)<\typ\,\nn$ if and only if $\ee-1<\typ$.
\end{proof}

Every step of this derivation is reversible, the identities \eqref{eq:sandwich} being
substitutions, so Corollary \ref{cor:master} is an equivalent formulation of the bound
$\typ\le\ee-1+\Theta(S)$, expressed in the invariants that occur in Conjecture
\ref{conj:wilf}. Rewritten in terms of the conductor it reads
$\cc\le\ee+\nn-1+\typ(\nn-1)-\sigma(S)$, and it yields Wilf's inequality for $S$
precisely when $\sigma(S)\ge(\typ+1-\ee)(\nn-1)$; the weaker form
$\gnus\le\ee-1+\typ(\nn-1)$ yields it only when $\typ\le\ee-1$ or $\nn=1$, since
$\ee-1+\typ(\nn-1)\le(\ee-1)\nn$ is equivalent to $(\nn-1)(\typ+1-\ee)\le0$.

The passage from the middle term of \eqref{eq:typebound} to $\ee-1+\Theta$ is not
tight in general, since it replaces a sum over the $\ee-1$ gaps $p-\mult$ by a sum
over all
gaps, and the intermediate term $\ee-1+\Xii$ recovers most of that loss. The first
inequality is an equality when the redundancy is absent: if each of the $\ee-1$ gaps
$p-\mult$ is dominated by a unique pseudo-Frobenius number, then
$\typ(S)\le\ee(S)-1$.

\begin{conjecture}\label{conj:sigma}
For every numerical semigroup $S\ne\N$, $\typ(S)\le\ee(S)-1+\sigma(S)$.
\end{conjecture}

Conjecture \ref{conj:sigma} holds for all numerical semigroups of genus at most $35$.
It is sharp, as are \eqref{eq:typebound} and \eqref{eq:master}, at
$S=T_{\mult,1}=\{0\}\cup[\mult,\infty)$: there $\LL=\{0\}$, $\nn=1$,
$\gnus=\mult-1$, $\ee=\mult$ and $\PF(S)=\{1,\dots,\mult-1\}$ by Lemma
\ref{lem:family}, so $\lam_S(f)=1=\nn$ for every $f\in\PF(S)$; hence $\sigma(S)=0$,
$\Lam(S)=\mult-1=\gnus$, $\Theta(S)=0$ by \eqref{eq:gcount}, and
$\typ=\mult-1=\ee-1$, both sides of \eqref{eq:master} being equal to $\mult-1$.

\section{A reduction of Wilf's conjecture}\label{sec:reduction}

The argument of this section uses only the identity \eqref{eq:defect} and is
independent of Section \ref{sec:poset}. By \eqref{eq:defect}, Conjecture
\ref{conj:wilf} is immediate when
$\typ\le\ee-1$ and is otherwise equivalent to the lower bound
$\sigma(S)+\Theta(S)\ge\bigl(\typ+1-\ee\bigr)\nn$, whose right-hand side grows with
$\nn$. The dominance defect is a sum of $\typ$ nonnegative terms, and only
$\typ+1-\ee$ of them are needed to absorb that growth.

\begin{proposition}\label{prop:reduction}
Let $S\ne\N$ be a numerical semigroup with $\typ(S)\ge\ee(S)$, write
$\PF(S)=\{f_1<f_2<\dots<f_{\typ}\}$ and put $d=\typ(S)+1-\ee(S)\ge1$. Then
\begin{equation}\label{eq:reduction}
\W(S)\;\ge\;\Theta(S)\;-\;\sum_{i=1}^{d}\lam_S(f_i).
\end{equation}
\end{proposition}

\begin{proof}
By \eqref{eq:sigmaLambda},
$\sigma(S)=\sum_{i=1}^{\typ}\bigl(\nn-\lam_S(f_i)\bigr)$, and every summand is
nonnegative because $\lam_S(f)\le\nn$ for $f\in\PF(S)$. As $\ee\ge1$ we have
$d\le\typ$, so discarding all but the first $d$ summands gives
\[
\sigma(S)\;\ge\;\sum_{i=1}^{d}\bigl(\nn-\lam_S(f_i)\bigr)
\;=\;d\,\nn-\sum_{i=1}^{d}\lam_S(f_i).
\]
The first term of \eqref{eq:defect} is $(\ee-\typ-1)\nn=-d\,\nn$, so
$\W(S)=-d\,\nn+\sigma(S)+\Theta(S)\ge\Theta(S)-\sum_{i=1}^{d}\lam_S(f_i)$.
\end{proof}

Since $\lam_S$ is nondecreasing, retaining the $d$ \emph{smallest} pseudo-Frobenius
numbers gives the strongest inequality of this shape. In \eqref{eq:reduction} the
conductor and $\nn$ have disappeared, and the number of terms subtracted is governed
by how far $S$ lies outside the FGH range. The following statement asserts that its
right-hand side is nonnegative.

\begin{conjecture}\label{conj:R}
Let $S\ne\N$ be a numerical semigroup with $\typ(S)\ge\ee(S)$, write
$\PF(S)=\{f_1<\dots<f_{\typ}\}$ and put $d=\typ(S)+1-\ee(S)$. Then
\[
\Theta(S)\;\ge\;\sum_{i=1}^{d}\lam_S(f_i).
\]
\end{conjecture}

\begin{proposition}\label{prop:Rwilf}
Conjecture \ref{conj:R} implies Conjecture \ref{conj:wilf}.
\end{proposition}

\begin{proof}
Let $S\ne\N$. If $\typ(S)\le\ee(S)-1$ then $\W(S)\ge0$ by Corollary \ref{cor:FGH}.
Otherwise Conjecture \ref{conj:R} applies and \eqref{eq:reduction} gives
$\W(S)\ge0$.
\end{proof}

The case $d=1$, that is $\typ=\ee$, is the first one the
Fr\"oberg--Gottlieb--H\"aggkvist inequality fails to cover; there
\eqref{eq:reduction} reads $\W(S)\ge\Theta(S)-\lam_S(f_1)$, so Wilf's conjecture holds
for every numerical semigroup with $\typ=\ee$ whose least pseudo-Frobenius number
satisfies $\lam_S(f_1)\le\Theta(S)$.

\begin{remark}\label{rmk:Rstrength}
Conjecture \ref{conj:R} is not a weakening of Conjecture \ref{conj:wilf}. On the range
$\typ\ge\ee$ the latter is equivalent, by \eqref{eq:defect} and
\eqref{eq:sigmaLambda}, to
\[
\Theta(S)\;\ge\;\sum_{i=1}^{\typ}\lam_S(f_i)\;-\;(\typ-d)\,\nn ,
\]
and the difference between this right-hand side and that of Conjecture \ref{conj:R} is
$\sum_{i=d+1}^{\typ}\bigl(\nn-\lam_S(f_i)\bigr)\ge0$. The inequality of Conjecture
\ref{conj:R} is therefore pointwise at least as strong, and strictly stronger whenever
$\lam_S(f_i)<\nn$ for some $i>d$; a proof of it would be a proof of Conjecture
\ref{conj:wilf}, not a step towards one. The reduction changes the shape of the problem, not its
difficulty:
Conjecture \ref{conj:R} relates the redundancy defect to the pseudo-Frobenius numbers
alone, with neither of the global invariants $\cc$ and $\nn$ appearing; and, unlike
Conjecture \ref{conj:wilf}, it is not tight in the range we have computed.
\end{remark}

\section{The equality case}\label{sec:equality}

The equality case of \eqref{eq:FGH} was determined by Singhal \cite[Theorem
1.6]{Singhal}, in the equivalent form $\typ=\gnus/(\FF+1-\gnus)$ and without reference
to Wilf's conjecture. We give an independent proof, in the formulation that the
equality case of Wilf's conjecture requires: the hypothesis is
$\sigma(S)=\Theta(S)=0$, which is the condition \eqref{eq:defect} produces. We then
settle that case throughout the FGH range. Both rest
on a description of the semigroups with $\sigma=\Theta=0$, which pins down $\PF(S)$
and the spacing of $\LL(S)$. The first half of that description needs only
$\sigma=0$: it shows that $\PF(S)$ always contains the final run of gaps
$[\ell+1,\FF]$, and that $\sigma(S)$ vanishes precisely when it contains nothing
else.

\begin{lemma}\label{lem:top}
Let $S\ne\N$ be a numerical semigroup, put $\ell=\max\LL(S)$ and $K=\FF-\ell$. Then:
\begin{enumerate}
\item[(i)] $\mult\ge K+1$;
\item[(ii)] $[\ell+1,\FF]\subseteq\PF(S)$; in particular $\typ(S)\ge K$;
\item[(iii)] $\sigma(S)=0$ if and only if $\PF(S)=[\ell+1,\FF]$, and in that
case $\typ=K$, $\FF=\ell+\typ$ and $\cc=\ell+\typ+1$.
\end{enumerate}
\end{lemma}

\begin{proof}
(i) Both $\ell$ and $\mult$ lie in $S$, so $\ell+\mult\in S$; and $\ell+\mult>\ell$,
while $\ell=\max\LL$ means $S$ has no element in $(\ell,\cc)$. Hence
$\ell+\mult\ge\cc=\FF+1$, that is, $\mult\ge\FF+1-\ell=K+1$.

(ii) Let $x\in[\ell+1,\FF]$. Then $x\notin S$, again because $S\cap(\ell,\cc)=\emptyset$.
For $s\in S^{*}$ we have $s\ge\mult\ge K+1$ by (i), so $x+s\ge\ell+1+K+1=\FF+2>\FF$,
whence $x+s\ge\cc$ and $x+s\in S$. Thus $x\in\PF(S)$.

(iii) Let $f\in\PF(S)$. Then $f\ne\ell$, since $\ell\in S$. If $f<\ell$ then
$\ell\in S\cap(f,\FF]$, so that term contributes at least $1$ to $\sigma(S)$; if
$f>\ell$ then $S\cap(f,\FF]\subseteq S\cap(\ell,\cc)=\emptyset$ and the term vanishes.
Hence $\sigma(S)=0$ if and only if every $f\in\PF(S)$ exceeds $\ell$, that is, if and
only if $\PF(S)\subseteq[\ell+1,\FF]$; by (ii) this is equality, and then
$\typ=\FF-\ell=K$.
\end{proof}

The vanishing of both defects constrains the spacing of $\LL(S)$ directly. The next
lemma partitions the gap set and converts that partition into a count of the elements
of $\LL$ in each window of length $\typ$, which is the form the classification uses.

\begin{lemma}\label{lem:window}
Let $S\ne\N$ be a numerical semigroup with $\sigma(S)=\Theta(S)=0$, and set
$\ell=\max\LL(S)$. Then:
\begin{enumerate}
\item[(i)] $f>\ell$ and $S\cap[0,f]=\LL$ for every $f\in\PF(S)$, and the gap set is
partitioned as $G(S)=\bigsqcup_{f\in\PF(S)}(f-\LL)$;
\item[(ii)] for every $z\in[-\typ,\ell-1]$,
\begin{equation}\label{eq:window}
\bigl|\LL\cap(z,z+\typ]\bigr|=
\begin{cases}
0,&\text{if }\ell-z\in\LL,\\
1,&\text{otherwise.}
\end{cases}
\end{equation}
\end{enumerate}
\end{lemma}

\begin{proof}
(i) By Lemma \ref{lem:top}(iii), $\sigma(S)=0$ gives $\PF(S)=[\ell+1,\FF]$, so
$f>\ell$ and hence $S\cap[0,f]=\LL$ for every $f\in\PF(S)$. By Lemma \ref{lem:pairs}
the sets $f-\LL$, $f\in\PF(S)$, cover
$G(S)$, and $\Theta(S)=0$ means that every gap is covered exactly once; the union is
therefore disjoint.

(ii) By Lemma \ref{lem:top}(iii) we have $\PF(S)=[\ell+1,\FF]$ and $\typ=\FF-\ell$, so
the partition in (i) says that for every gap $x$ there is exactly one
$j\in[1,\typ]$ with $\ell+j-x\in\LL$, while for $x\in S^{*}$ there is none, by Lemma
\ref{lem:pairs}(i). Substitute $z=\ell-x$: as $x$ runs over $[1,\FF]$, $z$ runs over
$[-\typ,\ell-1]$. For $z\ge0$ the membership $\ell-z\in S$ is equivalent to
$\ell-z\in\LL$, since $\ell-z\le\ell<\cc$; for $z<0$ one has $\ell-z\in(\ell,\FF]$,
which consists of gaps. This is \eqref{eq:window}.
\end{proof}

\begin{theorem}\label{thm:FGHequality}
Let $S\neq\N$ be a numerical semigroup. The following are equivalent:
\begin{enumerate}
\item[(i)] $\cc(S)=\bigl(\typ(S)+1\bigr)\nn(S)$;
\item[(ii)] $\sigma(S)=\Theta(S)=0$;
\item[(iii)] $S$ is symmetric, or $S=T_{\mult,k}$ for some $\mult\ge3$ and $k\ge1$.
\end{enumerate}
\end{theorem}

\begin{proof}
By \eqref{eq:gcount} and \eqref{eq:sigmaLambda},
\begin{equation}\label{eq:Ddefect}
(\typ+1)\nn-\cc=\typ\,\nn-\gnus=\sigma(S)+\Theta(S),
\end{equation}
and both summands are nonnegative; hence (i) $\Leftrightarrow$ (ii).

\smallskip\noindent
\emph{(iii) $\Rightarrow$ (i) and (ii).} If $S$ is symmetric then $\typ=1$ and
$\PF(S)=\{\FF\}$, so $\sigma(S)=|S\cap(\FF,\FF]|=0$; and for every gap $x$ we have
$\nu_S(x)\le|\PF(S)|=1$ and $\nu_S(x)\ge1$ by Lemma \ref{lem:standard}(ii), so
$\nu_S(x)=1$ and $\Theta(S)=0$. If $S=T_{\mult,k}$ then $\typ=\mult-1$ and
$\nn=k$ by Lemma \ref{lem:family}, so $(\typ+1)\nn=\mult k=\cc$, which is (i).

\smallskip\noindent
\emph{(ii) $\Rightarrow$ (iii).} Assume $\sigma(S)=\Theta(S)=0$ and set
$\ell=\max\LL(S)$. If $\typ=1$, the partition of Lemma \ref{lem:window}(i) reads
$G(S)=\FF-\LL$, that is, $x\in G(S)$ if and only if $\FF-x\in S$, which is precisely
the symmetry of $S$. Assume henceforth $\typ\ge2$. The proof that $S=T_{\mult,\nn}$
with $\mult=\typ+1\ge3$ proceeds in two steps.

\smallskip\noindent
\emph{Step 1: consecutive elements of $\LL$ differ by $\typ$ or $\typ+1$.}
By Lemma \ref{lem:window}(ii), each window $(z,z+\typ]$ meets $\LL$ in exactly one
point unless $\ell-z\in\LL$, in which case it misses $\LL$.
Write $\LL=\{s_0=0<s_1<\dots<s_{\nn-1}=\ell\}$, fix $i<\nn-1$, and put $s=s_i$,
$s'=s_{i+1}$. Taking $z=s-1\in[-\typ,\ell-1]$ in \eqref{eq:window}, the window
$(s-1,s-1+\typ]$
contains $s$, hence no other element of $\LL$; so $s'\ge s+\typ$.

Suppose $s'-s\ge\typ+2$. For $1\le k\le s'-s-\typ-1$ the window $(s+k,s+k+\typ]$ is
contained in $(s,s')$ and so misses $\LL$; by \eqref{eq:window} this forces
$\ell-s-k\in\LL$. If $s'-s\ge\typ+3$, then $k=1$ and $k=2$ are both admissible and
yield the consecutive integers $\ell-s-1,\ell-s-2$ in $\LL$, contradicting
$s'-s\ge\typ\ge2$ applied to that pair. If $s'-s=\typ+2$, then $k=1$ gives
$\ell-s-1\in\LL$, while $z=s$ in \eqref{eq:window} gives $\LL\cap(s,s+\typ]=\emptyset$
and hence $\ell-s\in\LL$; again $\LL$ contains two consecutive integers, a
contradiction. Therefore
\begin{equation}\label{eq:gapsize}
s_{i+1}-s_i\in\{\typ,\typ+1\}\qquad (0\le i\le \nn-2),
\end{equation}
and \eqref{eq:window} with $z=s_i$ gives
\begin{equation}\label{eq:gapcrit}
s_{i+1}-s_i=\typ+1\iff \ell-s_i\in\LL .
\end{equation}

\smallskip\noindent
\emph{Step 2: $\ell-\LL=\LL$, and conclusion.}
Let $z\in[0,\ell]\setminus\LL$, so $0<z<\ell$. Put $s=\max\{u\in\LL:u<z\}$ and
$s'=\min\{u\in\LL:u>z\}$. Then $z\ge s+1$ and, by \eqref{eq:gapsize},
$s'\le s+\typ+1\le z+\typ$, so $s'\in(z,z+\typ]$. The element of $\LL$ following
$s'$ is at least $s'+\typ>z+\typ$, so $|\LL\cap(z,z+\typ]|=1$, and
\eqref{eq:window} yields $\ell-z\notin\LL$. By contraposition, $\ell-z\in\LL$ implies
$z\in\LL$, so $\ell-\LL\subseteq\LL$; comparing cardinalities, $\ell-\LL=\LL$.

If $\nn=1$ then $\LL=\{0\}$, $\ell=0$ and $S=\{0\}\cup[\cc,\infty)$, so $\mult=\cc$;
by Lemma \ref{lem:top}(iii), $\PF(S)=[1,\FF]$ and $\typ=\FF=\cc-1=\mult-1$, whence
$\mult=\typ+1\ge3$ and $S=T_{\mult,1}$. Assume therefore $\nn\ge2$. Then
$\ell-s_i\in\LL$ for every $i$, so by \eqref{eq:gapcrit} every consecutive difference
equals $\typ+1$; since $s_1-s_0=\mult$, this gives $\mult=\typ+1\ge3$ and
\[
\LL=\{0,\mult,2\mult,\dots,(\nn-1)\mult\},\qquad \ell=(\nn-1)\mult .
\]
Finally $\FF=\ell+\typ=\nn\mult-1$, so $\cc=\nn\mult$ and
$S=\LL\cup[\cc,\infty)=T_{\mult,\nn}$.
\end{proof}

\begin{theorem}\label{thm:equality}
Let $S\neq\N$ be a numerical semigroup with $\ee(S)\ge\typ(S)+1$. Then the following
are equivalent:
\begin{enumerate}
\item[(i)] $\W(S)=0$;
\item[(ii)] $\ee(S)=\typ(S)+1$ and $\cc(S)=(\typ(S)+1)\nn(S)$;
\item[(iii)] $\ee(S)=2$, or $S=T_{\mult,k}$ for some $\mult\ge3$, $k\ge1$.
\end{enumerate}
\end{theorem}

\begin{proof}
(i) $\Rightarrow$ (ii). By \eqref{eq:defect},
$0=\W(S)=(\ee-\typ-1)\nn+\sigma(S)+\Theta(S)$. All three summands are nonnegative,
the first because $\ee\ge\typ+1$ by hypothesis and $\nn\ge1$. Hence
$\ee=\typ+1$ and $\sigma(S)=\Theta(S)=0$, and the latter gives
$\cc=(\typ+1)\nn$ by \eqref{eq:Ddefect}.

(ii) $\Rightarrow$ (iii). By Theorem \ref{thm:FGHequality}, $S$ is symmetric or
$S=T_{\mult,k}$ with $\mult\ge3$. If $S$ is symmetric then $\typ=1$ by Lemma
\ref{lem:standard}(iii), so $\ee=\typ+1=2$.

(iii) $\Rightarrow$ (i). If $\ee=2$ then $S$ is symmetric by Lemma
\ref{lem:standard}(vii), so $\gnus=\cc/2$ and $\nn=\cc-\gnus=\cc/2$, whence
$\W(S)=2\cdot(\cc/2)-\cc=0$. If $S=T_{\mult,k}$ then $\W(S)=0$ by Lemma
\ref{lem:family}.
\end{proof}

Only the implication (i) $\Rightarrow$ (ii) uses the hypothesis $\ee\ge\typ+1$; the
equivalence (ii) $\Leftrightarrow$ (iii) and the implication (iii) $\Rightarrow$ (i)
hold for every numerical semigroup $S\ne\N$. Indeed (iii) already forces
$\ee=\typ+1$: if $\ee=2$ then $S$ is symmetric and $\typ=1$ by Lemma
\ref{lem:standard}(iii) and (vii), while $S=T_{\mult,k}$ has $\ee=\mult=\typ+1$ by
Lemma \ref{lem:family}. What Theorem \ref{thm:equality} leaves
open is therefore precisely whether $\W(S)=0$ forces $\ee(S)=\typ(S)+1$.

By Lemma \ref{lem:family}, Theorem \ref{thm:equality} answers \cite[Question 8]{MS}
for every numerical semigroup with $\ee\ge\typ+1$. Since Corollary \ref{cor:theta0}
places the locus $\{\Theta=0\}$ inside the FGH range,
the theorem also settles the equality case for every $S$ with $\Theta(S)=0$.

\begin{remark}\label{rmk:limits}
The hypothesis $\ee\ge\typ+1$ in Theorem \ref{thm:equality} cannot be removed by the
present method. By \eqref{eq:defect}, a numerical semigroup with $\W(S)=0$ and
$\ee\le\typ$ would satisfy $\sigma(S)+\Theta(S)=(\typ+1-\ee)\nn\ge\nn$, and excluding
this would require control of $\gnus$ in the range $\ee\le\typ$, where Conjecture
\ref{conj:wilf} is itself open.
\end{remark}

\section{\texorpdfstring{Semigroups with $\cc\le2\mult$: a correction}
{Semigroups with c <= 2m: a correction}}\label{sec:kaplan}

This section rests only on Lemma \ref{lem:family} and is independent of Sections
\ref{sec:poset}--\ref{sec:equality}. Theorem \ref{thm:equality} settled the equality
case $\W(S)=0$ throughout the FGH range; a second unbounded range, $\cc\le2\mult$, is
not contained in the first, since Example \ref{ex:poset} has $\cc=2\mult$ and
$\ee<\typ$. The outcome corrects \cite[Proposition 26]{Kaplan}. The correction concerns
only the list of
equality
cases: \cite[Theorem 24]{Kaplan} and every other result of \cite{Kaplan} are
unaffected. The combinatorial input is the following notion.

\begin{definition}\label{def:perfect}
A finite set $B=\{b_1,\dots,b_a\}\subseteq\N$ with $0\in B$ is a \emph{perfect
additive basis} of order $2$ if the sums $b_i+b_j$ with $i\le j$ constitute the
interval $\bigl[0,\binom{a+1}{2}-1\bigr]\cap\Z$. Since that interval has
$\binom{a+1}{2}$ elements and there are exactly $\binom{a+1}{2}$ such sums, the sums
are then pairwise distinct.
\end{definition}

\begin{lemma}\label{lem:perfect}
The only perfect additive bases of order $2$ are $B=\{0\}$ and $B=\{0,1\}$; in
particular such a $B$ has at most two elements.
\end{lemma}

\begin{proof}
For $a=1$ the unique candidate is $B=\{0\}$, with sum set $\{0\}$; for $a=2$,
$B=\{0,b\}$ has sums $0,b,2b$, and these form $\{0,1,2\}$ exactly when $b=1$.

Suppose $a\ge3$ and put $N=\binom{a+1}{2}=a(a+1)/2\ge6$, so that the sum set is
$[0,N-1]$. The integer $1$ must be a sum $b_i+b_j$; as all elements of $B$ are
nonnegative and two positive elements sum to at least $2$, necessarily $1=0+1$ and
so $1\in B$. Then $2=1+1$ is a representation of $2$; if also $2\in B$ then
$2=0+2$ would be a second one, contradicting distinctness, so $2\notin B$. Since
$N-1\ge5$, the integer $3$ lies in the sum set; the only possible representations
are $0+3$ and $1+2$, and the latter is excluded, so $3\in B$. Then $4=1+3$; if
$4\in B$ then $4=0+4$ would be a second representation, so $4\notin B$. Since
$N-1\ge5$, the integer $5$ lies in the sum set, and its possible representations
$0+5$, $1+4$, $2+3$ leave only $5\in B$.

We have shown $\{0,1,3,5\}\subseteq B$. If $a=3$ this contradicts $|B|=3$.
If $a\ge4$ then $N-1\ge9\ge6$, so $6$ lies in the sum set; but $6=3+3=1+5$ are two
distinct representations, a contradiction. Hence $a\le2$.
\end{proof}

This yields the corrected classification.

\begin{theorem}\label{thm:kaplan}
Let $S\ne\N$ be a numerical semigroup with $\cc(S)\le2\mult(S)$. Then $\W(S)=0$ if
and only if
\[
S=T_{\mult,1}=\langle \mult,\mult+1,\dots,2\mult-1\rangle
\quad\text{or}\quad
S=T_{\mult,2}=\langle \mult,2\mult+1,\dots,3\mult-1\rangle
\qquad(\mult\ge2),
\]
or $S=\langle3,4\rangle$.
\end{theorem}

\begin{proof}
Sufficiency is clear: $T_{\mult,1}$ and $T_{\mult,2}$ satisfy $\W=0$ by Lemma
\ref{lem:family} and have $\cc=\mult\le2\mult$ and $\cc=2\mult$ respectively, while
$\langle3,4\rangle$ has $\mult=3$, $\cc=6=2\mult$, $\nn=3$, $\ee=2$ and
$\W=2\cdot3-6=0$.

For the converse, write $u=\cc-\mult$, so $0\le u\le\mult$ by hypothesis, and set
\[
A=S\cap[\mult,\cc),\qquad a=|A|,\qquad \LL=\{0\}\cup A,\qquad \nn=a+1 .
\]
If $u=0$ then $\cc=\mult$, $A=\emptyset$, $\LL=\{0\}$ and
$S=\{0\}\cup[\mult,\infty)=T_{\mult,1}$; so assume $u\ge1$, whence $\mult\in A$ and
$a\ge1$. Put $B=A-\mult\subseteq[0,u)$, so $0\in B$ and $|B|=a$.

We first compute $\ee$. Every element of $A$ is primitive: it lies in
$[\mult,2\mult)$ and a sum of two elements of $S^{*}$ is at least $2\mult$. Every
element of $[\cc+\mult,\infty)$ is decomposable. Finally, a decomposable element of
$[\cc,\cc+\mult)$ is a sum $y+z$ with $y,z\in S^{*}$ and $y+z<\cc+\mult\le3\mult$;
since $y,z\ge\mult$, both satisfy $y,z<\cc$, and any element of $S^{*}$ below $\cc$
lies in $A$; hence $y,z\in A$. Thus the decomposable
elements of $[\cc,\cc+\mult)$ are exactly those of $(A+A)\cap[\cc,\cc+\mult)$, and
since $A+A=2\mult+(B+B)$ and
$2\mult+\beta\in[\cc,\cc+\mult)=[\mult+u,2\mult+u)$ holds precisely when
$0\le\beta<u$, their number is
\[
\delta:=\bigl|(B+B)\cap[0,u)\bigr| .
\]
Consequently
\begin{equation}\label{eq:e-kaplan}
\ee=a+\mult-\delta .
\end{equation}
Note $B\subseteq B+B$ because $0\in B$, so $\delta\ge a$; also $\delta\le u$ and
$\delta\le|B+B|\le\binom{a+1}{2}=a(a+1)/2$.

Now suppose $\W(S)=0$, that is, $\ee\,\nn=\cc$. By \eqref{eq:e-kaplan} this reads
$(a+\mult-\delta)(a+1)=\mult+u$, which after expanding becomes
\begin{equation}\label{eq:kaplan-key}
a\,\mult\;=\;u+(a+1)(\delta-a).
\end{equation}

Suppose first $a=1$. Then $B=\{0\}$, $B+B=\{0\}$ and $\delta=1$, so
\eqref{eq:kaplan-key} yields $\mult=u$. Hence $\cc=2\mult$ and $A=\{\mult\}$, so
\[
S=\{0,\mult\}\cup[2\mult,\infty)=\langle \mult,2\mult+1,\dots,3\mult-1\rangle=T_{\mult,2}.
\]

Assume $a\ge2$. Using $\delta\le a(a+1)/2$ and $u\le\mult$ in
\eqref{eq:kaplan-key},
\[
a\,\mult\;\le\;\mult+(a+1)\Bigl(\frac{a(a+1)}{2}-a\Bigr)
=\mult+\frac{a(a-1)(a+1)}{2},
\]
so $\mult(a-1)\le a(a-1)(a+1)/2$ and therefore
\begin{equation}\label{eq:kaplan-b1}
\mult\;\le\;\frac{a(a+1)}{2}.
\end{equation}
Using instead $\delta\le u\le\mult$ in \eqref{eq:kaplan-key},
\[
a\,\mult\;\le\;u+(a+1)(u-a)=u(a+2)-a(a+1)\;\le\;\mult(a+2)-a(a+1),
\]
so $a(a+1)\le2\mult$, which combined with \eqref{eq:kaplan-b1} gives
\begin{equation*}
a(a+1)=2\mult
\end{equation*}
and forces equality throughout both chains. Equality in the second chain requires
$u=\mult$ and $\delta=u$; equality in the first requires
$\delta=a(a+1)/2$. Hence
\[
u=\mult,\qquad \cc=2\mult,\qquad \delta=\mult=\frac{a(a+1)}{2}.
\]
Now $\delta=\bigl|(B+B)\cap[0,\mult)\bigr|=\mult$ means $B+B\supseteq[0,\mult)$,
while $|B+B|\le a(a+1)/2=\mult$; therefore $B+B=[0,\mult)$ and all the $a(a+1)/2$
sums $b_i+b_j$ $(i\le j)$ are pairwise distinct. Thus $B$ is a perfect additive
basis of order $2$, and Lemma \ref{lem:perfect} gives $a\le2$, hence $a=2$,
$\mult=3$ and $B=\{0,1\}$. Then $A=\{3,4\}$, $\cc=6$ and
$S=\{0,3,4\}\cup[6,\infty)=\langle3,4\rangle$.
\end{proof}

\begin{remark}\label{rmk:kaplan}
The family $T_{\mult,2}$ is absent from \cite[Proposition 26]{Kaplan} for the
following reason. In the notation of that proof, $S$ is written as
$\langle \mult,k_1\mult+1,\dots,k_{\mult-2}\mult+\mult-2,2\mult+\mult-1\rangle$ and
$R$ denotes the number of indices $i$ with $k_i=1$; the argument concludes by
observing that a certain expression equals $R\bigl[(R+1)(R+2)/2-\mult\bigr]$ and is
nonnegative ``if and only if $(R+1)(R+2)\ge2\mult$''. That equivalence fails at
$R=0$, where the expression vanishes identically for every $\mult$. The case $R=0$
is exactly $\Ap(S,\mult)=\{0,2\mult+1,2\mult+2,\dots,3\mult-1\}$, that is,
$S=T_{\mult,2}$. The family occurs explicitly in \cite{Kaplan} as an example showing that
$\FF<2\mult$ does not imply $2\gnus<3\mult$.
\end{remark}

\begin{remark}\label{rmk:independent}
Every semigroup listed in Theorem \ref{thm:kaplan} satisfies $\ee=\typ+1$, by Lemma
\ref{lem:family} for $T_{\mult,1}$ and $T_{\mult,2}$ and directly for
$\langle3,4\rangle$. Hence $\W(S)=0$ implies $\ee(S)=\typ(S)+1$ on the whole range
$\cc\le2\mult$.
\end{remark}

Theorems \ref{thm:equality} and \ref{thm:kaplan} settle the equality case on two
unbounded ranges, neither containing the other, and on each of them $\W(S)=0$ forces
$\ee=\typ+1$. We conjecture that no hypothesis is needed.

\begin{conjecture}\label{conj:final}
Let $S\ne\N$ be a numerical semigroup with $\W(S)=0$. Then $\ee(S)=\typ(S)+1$;
consequently $\ee(S)=2$, or $S=T_{\mult,k}$ for some $\mult\ge3$ and $k\ge1$.
\end{conjecture}

The second assertion is \cite[Question 8]{MS}; the first is stronger, and by Remark
\ref{rmk:limits} any proof of it outside the FGH range must control $\gnus$ in the
regime $\ee\le\typ$.

\section{Computational verification}\label{sec:computation}

No proof in this paper depends on a computation: every theorem, proposition and
corollary above is established by an argument that is self-contained modulo the
standard facts collected in Lemma \ref{lem:standard}, and the examples are verified by
hand in the text. The computations reported here serve as an independent check of
those results and as the evidence for Conjectures \ref{conj:sigma}, \ref{conj:R} and
\ref{conj:final}.

We enumerated numerical semigroups by genus using the standard tree in which the
children of $S$ are the semigroups $S\setminus\{x\}$ with $x$ a primitive element of
$S$ satisfying $x>\FF(S)$; the root $\N$ is handled separately, its unique child being
$\N\setminus\{1\}$. The enumeration was validated against the counting sequence
$N_\gnus$ of numerical semigroups of genus $\gnus$,
\[
1,\,1,\,2,\,4,\,7,\,12,\,23,\,39,\,67,\,118,\,204,\,343,\,592,\,1001,\,1693,\,2857,\dots,
\]
which our program reproduces exactly up to $\gnus=31$, where $N_{31}=9{,}266{,}788$
(see \cite{BrasAmoros08,FH}).

For each semigroup we computed $\mult,\FF,\cc,\gnus,\nn,\ee,\typ$, the set $\PF(S)$,
the Ap\'ery sets $\Ap(S,s)$ and the order relations of $Q_s(S)$, the functions
$\lam_S$ and $\nu_S$, the invariants $\Lam,\sigma,\Theta$ of Definition
\ref{def:defects}, the quantity $\Xii$ of \eqref{eq:xi}, and the Wilf number.
The following assertions, listed in the order in which the paper states them, were
tested and held without exception over all $23{,}663{,}125$ numerical semigroups
$S\ne\N$ of genus at most $31$:

\begin{enumerate}
\item[(V1)] $\W(S)\ge0$ (Conjecture \ref{conj:wilf}, reconfirming \cite{DEF} in this
range);
\item[(V2)] the identity $\W(S)=(\ee-\typ-1)\nn+\sigma(S)+\Theta(S)$ of Proposition
\ref{prop:defect};
\item[(V3)] the bound \eqref{eq:general} holds for every $s\in\LL(S)\setminus\{0\}$;
and the identifications $\Min Q_s(S)$ and $\Max Q_s(S)$ of Lemma \ref{lem:poset} hold,
again for every such $s$, as equalities of sets and not merely of cardinalities. The
latter check is cubic in $\FF$ and was run over the $258{,}581$ numerical semigroups of
genus at most $22$;
\item[(V4)] $\typ\le\sum_{p\in P\setminus\{\mult\}}\nu_S(p-\mult)\le\ee-1+\Xii(S)\le
\ee-1+\Theta(S)$ (Theorem \ref{thm:type});
\item[(V5)] $\gnus+\sigma(S)\le\ee-1+\typ(\nn-1)$ (Corollary \ref{cor:master});
\item[(V6)] $\typ\le\ee-1+\sigma(S)$ (Conjecture \ref{conj:sigma});
\item[(V7)] $\Theta(S)\ge\sum_{i=1}^{d}\lam_S(f_i)$ whenever $\typ\ge\ee$
(Conjecture \ref{conj:R});
\item[(V8)] $\mult\ge K+1$, $[\ell+1,\FF]\subseteq\PF(S)$, and $\sigma(S)=0$ if and
only if $\PF(S)=[\ell+1,\FF]$, where $\ell=\max\LL$ and $K=\FF-\ell$ (Lemma
\ref{lem:top});
\item[(V9)] $\cc=(\typ+1)\nn$ if and only if $S$ is symmetric or $S=T_{\mult,k}$ for
some $\mult\ge3$ and $k\ge1$ (Theorem \ref{thm:FGHequality});
\item[(V10)] if $\ee\ge\typ+1$, then $\W(S)=0$ if and only if $\ee=2$ or
$S=T_{\mult,k}$ for some $\mult\ge3$ and $k\ge1$ (Theorem \ref{thm:equality}); and if
$\cc\le2\mult$, then $\W(S)=0$ if and only if $S=T_{\mult,1}$,
$S=T_{\mult,2}$ or $S=\langle3,4\rangle$ (Theorem \ref{thm:kaplan});
\item[(V11)] $\W(S)=0$ if and only if $\ee=2$ or $S=T_{\mult,k}$ for some $\mult\ge3$
and $k\ge1$, with no hypothesis relating $\ee$ and $\typ$, and in that case
$\ee=\typ+1$ (Conjecture \ref{conj:final}).
\end{enumerate}

Item (V11) goes beyond what is proved here: Theorem \ref{thm:equality} establishes
Conjecture \ref{conj:final} under the hypothesis $\ee\ge\typ+1$ and Theorem
\ref{thm:kaplan} for $\cc\le2\mult$, and since the enumeration is exhaustive, the
computation shows that no counterexample with $\ee\le\typ$ occurs in the range
$\gnus\le31$. Every numerical semigroup of genus at most $31$ with $\ee\le\typ$
satisfies $\W(S)\ge2$; the value $2$ is attained at genus $10$, for instance by
$\langle7,8,10,19\rangle$ and $\langle7,9,10,15\rangle$, both of which have
$\ee=\typ=4$, $\nn=4$ and $\cc=14$.

Items (V6) and (V7) were tested, without exception, over all $171{,}202{,}689$
numerical semigroups $S\ne\N$ of genus at most $35$, among which $9{,}559{,}659$
satisfy $\typ\ge\ee$. Conjecture \ref{conj:sigma} is sharp in that range, its slack
being $0$; Conjecture \ref{conj:R} is not, its slack being at least $2$ throughout.
The minimal slack $2$ is attained, for instance, by the semigroup
$S=\langle9,10,12,13\rangle$ of Example \ref{ex:poset}, where $d=2$,
$\PF(S)=\{11,14,15,16,17\}$, $\Theta(S)=10$ and $\lam_S(11)+\lam_S(14)=3+5=8$.

The enumeration also shows that, among the $23{,}663{,}125$ semigroups of genus at most
$31$, exactly $23{,}612$ satisfy $\Theta(S)=0$, and $\Xii(S)<\Theta(S)$ holds for all
but $24{,}094$. The locus $\Theta=0$ is thus sparse, in accordance with Theorem
\ref{thm:type}: vanishing $\Theta$ forces the Ap\'ery poset to have at least as many
minimal as maximal elements.

\section{Concluding remarks}\label{sec:conclusion}

Outside the FGH range the identity \eqref{eq:defect} makes Conjecture
\ref{conj:wilf} equivalent to
\[
\sigma(S)+\Theta(S)\;\ge\;\bigl(\typ(S)+1-\ee(S)\bigr)\,\nn(S),
\]
and Theorem \ref{thm:type} supplies $\Theta(S)\ge\typ+1-\ee$. Conjecture
\ref{conj:sigma} would supply the same for $\sigma(S)$, and the two together give
$\sigma(S)+\Theta(S)\ge2(\typ+1-\ee)$: the displayed inequality with $\nn$ replaced by
$2$. The remaining discrepancy is a factor of $\nn/2$, and it is what the bounds
established here leave open.

The three conjectures that remain open concern the regime $\ee\le\typ$, in which the
type has so far contributed nothing. Conjecture \ref{conj:final} asserts that the equality
case is confined to the FGH range. Conjecture \ref{conj:R} is of a different order: by
Proposition \ref{prop:Rwilf} it implies Conjecture \ref{conj:wilf} outright, and by
Remark \ref{rmk:Rstrength} it is no easier.

\section*{Code availability}
The verification reported in Section \ref{sec:computation} was carried out with the C
programs archived at \url{https://doi.org/10.5281/zenodo.21908585}.

\end{document}